\documentclass[12pt,twoside]{amsart}
\usepackage{amssymb}
\usepackage{a4wide}
\usepackage{amssymb}
\usepackage{a4wide}
\usepackage[latin1]{inputenc}
\usepackage[T1]{fontenc}
\usepackage{times}

\newtheorem{thm}{Theorem}[section]

\newtheorem{lem}[thm]{Lemma}
\newtheorem{cor}[thm]{Corollary}
\newtheorem{prop}[thm]{Proposition}
\newtheorem{defin}[thm]{Definition}
\newtheorem{definlem}[thm]{Definition--Lemma}

\def\bclaim{\begin{claim}}
\def\eclaim{\end{claim}}
\def\bdefin{\begin{defin}}
\def\edefin{\end{defin}}
\def\bcor{\begin{cor}}
\def\ecor{\end{cor}}
\def\bthm{\begin{thm}}
\def\ethm{\end{thm}}
\def\blem{\begin{lem}}
\def\elem{\end{lem}}
\def\bdeflem{\begin{definlem}}
\def\edeflem{\end{definlem}}
\def\blemma{\begin{lem}}
\def\elemma{\end{lem}}
\def\bprop{\begin{prop}}
\def\eprop{\end{prop}}
\def\bremark{\begin{remark}}
\def\eremark{\end{remark}}
\def\bpf{\begin{proof}}
\def\epf{\end{proof}}
\def\beq{\begin{equation}}
\def\eeq{\end{equation}}
\def\beqno{\begin{equation*}}
\def\eeqno{\end{equation*}}
\def\eaeq{\end{aligned}}
\def\baeq{\begin{aligned}}

\def\hpq0{h^{p,q}_{\leq 0}}
\def\Hpq0{\H_{\leq 0}^{p,q}}

\def\dbar{\bar\partial}
\def\ddbar{\partial\dbar}

\def\R{{\mathbb R}}

\def\C{{\mathbb C}}

\def\L{{L}}

\def\K{{\mathcal K}}

\def\H{{\mathcal H}}

\def\O{{\mathcal O}}

\def\Re{{\rm Re\,  }}

\def\U{{\mathcal U}}

\def\be{\begin{equation}}
\def\ee{\end{equation}}

\def\L{\mathcal{L}}
\def\opcit{\underbar{\phantom{aaaaa}}}

\def\MA{Monge--Amp\`ere }
\def\K{K\"ahler }
\def\i{\sqrt{-1}}
\def\del{\partial}
\def\dbar{\bar\partial}
\def\ddbar{\del\dbar}
\def\ra{\rightarrow}
\def\eps{\epsilon}
\newcommand{\RR}{\mathbb{R}}
\newcommand{\CC}{\mathbb{C}}

\def\del{\partial}
\newcommand{\calH}{\mathcal{H}}

\def\sm{\setminus}

\def\O{\Omega}
\def\sm{\setminus}
\def\w{\wedge}
\def\o{\omega}

\def\vp{\varphi}
\def\beq{\begin{equation}}
\def\eeq{\end{equation}}
\def\bi#1{\bibitem{#1}}

\def\h#1{\hbox{#1}}
\def\b#1{\bar{#1}}

\def\calL{{\mathcal L}}

\def\bal{\begin{aligned}}
\def\eal{\end{aligned}}
\def\q{\quad}

\font\Bbbsmall=msbm7

\def\outlinsmall#1{\hbox{\Bbbsmall #1}}

\def\CCfoot{{\outlinsmall C}}
\def\ovp{{\o_{\varphi}}}

\def\phiC{\phi_{\CCfoot}}
\def\Dphi{D_\phi}

\def\Lphi{\calL_\phi}

\def\La{\Lambda}
\def\la{\lambda}

\theoremstyle{definition}

\theoremstyle{remark}

\newtheorem{preremark}{Remark}
\newtheorem{preex}{Example}

\newenvironment{remark}{\begin{preremark}}{\qed\end{preremark}}

\def\bi#1{\bibitem{#1}}
\def\opcit{\underbar{\phantom{aaaaa}}}
\def\usc{\h{\rm usc}}

\numberwithin{equation}{section}

\begin{document}

\title[]
      {Complex Legendre duality}

\author[]{ Bo Berndtsson, Dario Cordero-Erausquin, Bo'az Klartag, Yanir A. Rubinstein}

\begin{abstract}

We introduce complex generalizations of the classical
Legendre transform,  operating on K\"ahler metrics on a compact complex manifold. These Legendre transforms give  explicit local isometric symmetries
for the Mabuchi metric on the space of K\"ahler metrics around any real analytic K\"ahler metric,
answering a question originating in Semmes' work.

\end{abstract}


\maketitle

\section{Introduction}
For a function $\psi:\R^N\ra\R$ its classical Legendre transform is defined as
\cite{Fenchel,Mandelbrojt}
\be
\psi^*(y)=\sup_x [x\cdot y-\psi(x)].
\ee
This transform plays an important role in several parts of mathematics,
notably in classical mechanics and convex geometry. Being the supremum of affine functions (of $y$), $\psi^*$ is always convex and in case $\psi$ is convex it equals the Legendre transform of its Legendre transform. One easily verifies
(see Section \ref{Sec2})  that $\psi^*=\psi$ if and only if $\psi(x)=x^2/2$, so the Legendre transform is a symmetry on the space of convex functions around its fixed point $x^2/2$.

In particular this applies when $\R^N=\C^n$. In this case we change the definition  slightly and put
\be
\hat\psi(w)=\sup_z[ 2\Re (z\cdot \bar w)-\psi(z)],
\ee
where $a\cdot b:=\sum_{i=1}^n a_i b_i$.
The reason for this change is that while the supremum in (1.1) is attained if $y=\partial \psi(x)/\partial x$, the supremum in (1.2) is attained if $w=\partial\psi(z)/\partial\bar z$, and one verifies that the unique fixed point is now $\psi(z)=|z|^2$.

In this connection a very interesting observation was made by Lempert, \cite{Lempert}:  Put $\omega_\psi:=i\ddbar\psi$. Assuming that $\psi$ is smooth and strictly convex, $g^*\omega_{\hat\psi}=\omega_\psi$, if $g(z)=\partial\psi(z)/\partial\bar z$. It follows from this that
\be
g^*\omega_{\hat\psi}^n=\omega_\psi^n.
\ee
The measure  $\omega_\psi^n/n!=:\h{MA}_\C(\psi)$ is the complex Monge--Amp\`ere measure associated to $\psi$, and Lempert's theorem thus implies that
$\h{MA}_\C(\psi)$ and $\h{MA}_\C(\hat\psi)$ are related under the gradient map $g$.
This may be compared and perhaps contrasted to the way real Monge--Amp\`ere measures transform under gradient maps, see Section~\ref{Sec2}.

At this point we recall the definition of the Mabuchi metric in the somewhat nonstandard setting of smooth, strictly plurisubharmonic functions on $\C^n$. The idea is to view this space as an infinite dimensional manifold; an open subset of the space of all smooth functions. Its tangent space at a point $\phi$ should consist of smooth functions $\chi$ such that $\phi+t\chi$ remains strictly plurisubharmonic for $t$ close enough to zero. This set of course depends on the particular $\phi$ we have chosen, but at any rate the tangent space will always contain smooth functions of compact support, so we take by definition the space of such functions as our tangent space. The Mabuchi norm of a tangent vector  $\chi$  at a point $\phi$ is now defined by
\be
\|\chi\|^2_\phi:= \int_{\C^n} |\chi|^2 \omega_\phi^n/n!.
\ee

We will interpret (1.3) as saying  that the Legendre transform is an isometry for the Mabuchi metric on the space of convex functions (see section 4). It follows, at least formally, that the Legendre transform maps geodesics for the Mabuchi metric to geodesics, which reflects the so called duality principle for the complex method of interpolation \cite{BBDY}.

This is only a 
special case of Lempert's result, which implies  that a much more general class of gradient-like maps are isometries for the Mabuchi metric. In this note we will develop this scheme and define `Legendre tranforms' for  K\"ahler potentials over a manifold $M$, usually compact.

For this we note first that the usual Legendre transform is not involutive on all plurisubharmonic functions, but just on convex functions, hence in particular on functions that are close to its fixed point $|z|^2$. Imitating this, we start with a (local) K\"ahler potential  $\phi$ on $M$, and define a 'Legendre tranform' depending on $\phi$ that is defined for potentials close to  $\phi$,  fixes $\phi$ and is an isometry for the Mabuchi metric. For this to work, we need to assume that $\phi$ is real analytic. The definition of the $\phi$-Legendre transform involves a polarization of our real analytic potential, which is locally a function $\phi_\C(z,w)$ defined near the diagonal in $M\times M$. $\phi_\C$ is
holomorphic in $z$, antiholomorphic in $w$ and coincides with $\phi$ on the diagonal (these properties determine $\phi_\C$ uniquely). Roughly speaking, the idea is then to replace $z\cdot \bar w$ by $\phi_\C$ and define our transform as
\be
(\L_\phi \psi)(w):=\sup_z [ 2\Re\phi_\C(z,w)-\psi(z)].
\ee
When $\phi(z)=|z|^2$ this gives us back the Legendre transform of (1.2). Let us first examine this transform 
in the case of a linear space, the cradle of the classical Legendre transform. Write $\Delta_{\CC^n} = \{ (z,z) \, ; \, z \in \CC^n \}$ for the diagonal. 
We say that a smooth function $\phi$
on $\CC^n$ is strongly plurisubharmonic if its complex Hessian is bounded from below by a positive constant, uniformly 
at all points of $\CC^n$. 

\begin{thm}
Let $\phi: \CC^n \rightarrow \RR$ be a real-analytic, strongly plurisubharmonic function.
Then there are an open set $V_\phi \subseteq \CC^n \times \CC^n$ containing $\Delta_{\CC^n}$
and a neighborhood $U$ of $\phi$ in the $C^2$-norm on $\CC^n$ with the following properties: 

0. For $u \in U$ the function $\L_{\phi}(u): \CC^n \rightarrow \RR \cup \{ + \infty \}$ is well-defined, where the supremum in (1.5) 
runs over all $z$ with $(z,w) \in V_{\phi}$.

1. $\L_\phi(u)=u$ if and only if $u=\phi$.

2. For $u$ in a smaller neighbourhood $U'\subset U$, the function $\L_\phi(u)$ lies in $U$ and $L_\phi^2(u)=u$.

3. $\L_\phi$ is an isometry for the Mabuchi metric restricted to $U$.
\end{thm}

The transform in (1.5) works fine if $\phi$, $\psi$ and $\phi_\C$ are defined on all of $\C^n$ or $\C^{2n}$, but for functions that are only locally well defined we need to find a variant of the definition that has a global meaning on a manifold.
For this it turns out to be very convenient to use a remarkable idea of Calabi, \cite{Calabi}. The Calabi {\it diastasis function} is defined as
$$
D_\phi(z,w)=\phi(z)+\phi(w)-2\Re\phi_{\C}(z,w).
$$
We then  change the above definition by applying it  to $\psi+\phi$ instead of $\psi$, and then subtract $\phi$ afterwards. This way we arrive at the equivalent transform
$$
L_\phi(\psi)(w):=\L_\phi(\psi+\phi)(w)-\phi(w)=\sup_z (-D_\phi(z,w) -\psi(z)).
$$
Notice that in the classical case when $\phi(z)=|z|^2$, $D_\phi(z,w)=|z-w|^2$ and the transform becomes the familiar variant of the Legendre transform
$$
\sup_z -(|z-w|^2 +\psi(z)).
$$
The point of this is that, as is well known from the work of Calabi, $D_\phi$ only depends on $\omega_\phi=i\ddbar\phi$, i e it does not change if we add a pluriharmonic function to $\phi$. As we shall see this implies that our construction of $L_\phi=L_{\omega_\phi}$ globalizes and becomes well defined on functions $\psi$ on a manifold $M$ that are close to 0 in the $C^2$-norm. Following the ideas, but not the precise proof, of Lempert, we can then verify that $L_{\o_\phi}$ is an isometry for the Mabuchi metric on $\U_{\omega_\phi}$. Our main result is as follows:

\begin{thm}
Let $M$ be a compact K\"ahler manifold, and let $\omega$ be a real analytic K\"ahler form on $M$. Let $\H_\omega:=\{u\in C^\infty(M); i\ddbar u+\omega>0\}$. Then the generalized Legendre transform, $L_\omega$ (defined in section 4) is defined on a neigbourhood $U$ of 0 in $\H_\omega$ in the $C^2$-toplogy  and

1. $L_\omega(u)=u$ if and only if $u=0$.

2. For $u$ in a smaller neighbourhood $U'\subset U$, $L_\omega(u)$ lies in $U$ and $L_\omega^2(u)=u$.

3. $L_\omega$ is an isometry for the Mabuchi metric on $\H_\omega$ restricted to $U$.
\end{thm}

\section{The classical Legendre transform}
\label{Sec2} As a warm up and for comparison we first briefly look at the classical Legendre transform, (1.1). If $\psi$ is differentiable, and if the supremum in the right hand side is attained  in a point $x$, then $y=\partial\psi/\partial x=:g_\psi(x)$. Hence we have that
\be
x\cdot y\leq \psi(x)+\psi^*(y)
\ee
with equality if and only if $y=g_\psi(x)$, or in other words $x\cdot g_\psi(x)-\psi(x)=\psi^*(g_\psi(x))$.
If moreover $\psi$ is assumed smooth and strictly convex, $g_\psi$ is invertible. It follows that $\psi^*$ is also smooth, and by the symmetry of (2.1) that the inverse of $g_\psi$ is $g_{\psi^*}$. Recall that the (real) Monge--Amp\`ere measure of a (smooth) convex function is $\h{MA}_\R(\psi):=\det(\psi_{j,k}(x))dx$. It follows from the above that
$$
g_\psi^*(dy)=\h{MA}_\R(\psi),\quad \mbox{and}\quad g_{\psi^*}^*(dx)=\h{MA}_\R(\psi^*).
$$

We next turn to the Legendre transform of functions on $\C^n$ and its relation to complex Monge--Amp\`ere measures. We then redefine the Legendre transform by (1.2). Equality now occurs when $w=\partial\psi(z)/\partial \bar z=:g_\psi$, where we have also redefined the gradient map $g$ to fit better with complex notation. We now give the first case of Lempert's theorem; it should be compared to how {\it real} Monge--Amp\`ere measures transform.
\begin{thm}
\label{gpsiThm}
{\rm (Lempert)} With the above notation
\be
g_\psi^*(\omega_{\hat\psi})=\omega_\psi,
\ee
so the complex Monge--Amp\`ere measures of $\psi$ and $\hat\psi$ are related by $g_\psi^*(\h{MA}_\C(\hat\psi))=\h{MA}_\C(\psi)$.
\end{thm}
\begin{proof} Let $\Lambda=\{(z,w); w=g_\psi(z)\}$ be the graph of the gradient map $g_\psi$ considered as a submanifold of $\C^{2n}$. On $\Lambda$
$$
d(z\cdot\bar w)=\partial\psi(z) +\dbar\hat\psi(w).
$$
(This is simply because when $(z,w)$  lie on $\Lambda$, then $\partial\psi(z)=\sum \bar w_j dz_j$ and $\dbar\hat\psi(w)=\sum z_j d\bar w_j$.)
Since the left hand side is a closed form, it follows that
$$
\bar\partial\partial\psi(z)=d\partial\psi(z)=-d\dbar\hat\psi(w)=-\ddbar\hat\psi(w).
$$
If we pull back this equation under the map $z\to (z,g_\psi(z))$ we get
$$
\ddbar\psi=g_\psi^*(\ddbar\hat\psi),
$$
which proves the theorem.
\end{proof}

We remark that the apparent discrepancy between how the gradient map transforms the real versus the complex Monge--Amp\`ere measures can be rectified as follows. First, since  $[\psi^{ij}]:=[\psi_{ij}]^{-1}=[\psi^*_{ij}]$
under appropriate regularity assumptions, the Riemannian metric
$\psi_{ij}dx^i\otimes dx^j$ is the pull-back of $\psi^*_{ij}dy^i\otimes dy^j$
via the gradient map $\nabla\psi:\R^n\ra\R^n$. Therefore
the gradient map pulls back the measure $\sqrt{\det[\psi^*_{ij}]}dy^1\w\cdots\w
dy^n$ to the measure $\sqrt{\det[\psi_{ij}]}dx^1\w\cdots\w
dx^n$. When $M$ has toric symmetry, Theorem \ref{gpsiThm}  precisely produces  this observation via a careful translation between the real notation and the complex notation
(cf. the proof of \cite[Proposition 2.1]{CR}).

\section{Complex Legendre transforms}

Let $\O$ denote a domain in $\CC^n$.
Denote by $\phi$ a real-analytic psh function on $\O$. 
Denote by $\phiC$ the analytic extension of $\phi$ to a holomorphic function on
a neighborhood $W_\phi$ of the diagonal in $\CC^n\times\overline{\CC^n}$
(whenever $(M,J)$ is a complex manifold we denote by $\overline M$
the complex manifold $(M,-J)$). Such an extension exists since the diagonal
$$
\Delta_{\CCfoot}:=\{(p,p)\in \CC^n\times \CC^n\,:\, p\in \CC^n\}
$$
is totally real in $\CC^n\times\overline{\CC^n}$.
We immediately switch point of view and work from now on in $\CC^n\times \CC^n$,
where $\phiC$ can be considered as a function on
$$
W_\phi\subset \CC^n\times \CC^n
$$
that is holomorphic in the first factor and anti-holomorphic
in the second. Explicitly, if in local coordinates
$$
\phi(z)=\sum c_{\alpha,\beta}z^\alpha \bar {z}^\beta,
$$
then 
$$
\phiC(z,w)=\sum c_{\alpha,\beta}z^\alpha \bar {w}^\beta.
$$

The
{\it Calabi diastasis function} associated to a real-analytic
strongly psh function $\phi$ on $\O\subset\CC^n$ %
is the function
\beq\label{DphiEq}
\Dphi(p,q):=
\phi(p)+\phi(q)
-\phiC(p,q)-\phiC(q, p)
=
\phi(p)+\phi(q)
-2\Re\phiC(p,q).
\eeq
defined on $W_\phi\subset \O\times\O$. 
Clearly $\Dphi(p,q) = \Dphi(q,p)$ with $\Dphi(p,p) = 0$. In the local coordinates, 
$$ \Dphi(p,q) =  \sum c_{\alpha,\beta} (p^{\alpha} - q^{\alpha}) \overline{(p^{\beta} - q^{\beta})}. $$
 Note that the first non-zero term in the Taylor series is non-negative 
 as $\phi$ is psh. Moreover, 
denote by $\pi_i:\CC^n\times \CC^n\ra \CC^n, \; i=1,2,$ the natural projections,
i.e., $\pi_1(z,w)=z, \pi_2(z,w)=w$.
Calabi proves the following \cite{Calabi}.
\blem
\label{CalabiLemma}
There exists an open neighborhood $V_\phi$
of $\Delta_{\O}$ contained in
$W_\phi$
on which
$\Dphi(\,\cdot\,,q)$
is strongly convex with 
\beq\label{StrictConvEq}
\Dphi(z,q)
\ge
C|z-q|^2
\q
\h{on\ }
\pi_1(V_{\phi}\cap \O\times\{q\}). 
\eeq

\elem


We can now define a Legendre type transform
associated to $\phi$.
For simplicity, whenever we refer to a function
in our discussions
below, we do not allow the constant function $+\infty$. We denote by
$\usc f$ the upper semi-continuous (usc) regularization of a function $f:X\ra \RR$,
$$
\usc f(x):=\lim_{\delta\ra0}\sup_{y\in X\atop |y-x|<\delta}f(y).
$$
It is the smallest usc function majorizing $f$.

\bdefin
\label{cxLegDef}
The complex Legendre transform $\Lphi$ is a mapping
taking a function $\psi:\O\ra\RR\cup\{\infty\}$ to
$$
\Lphi(\psi)(q)
:=
\usc\sup_{
p\in \O \atop (p,q)\in V_\phi
}
[2\Re \phiC(p,q)-\psi(p)].
$$
\edefin
Since $\Re \phiC(p,q)$ is pluriharmonic in $q$, $\Lphi$ is psh.
The definition depends on $V_\phi$, and we discuss that
dependence later.

When $\phi(z)=|z|^2$, then
$\phiC(z,w)=z\cdot \overline{w}$, while
$\Dphi(z,w)=|z-w|^2$, $C=1$ and $W_\phi=V_\phi=\CC^n\times\CC^n$; we recover, up to
a factor of 2, the Legendre
transform on $\RR^{2n}$.

\blem
\label{SelfDualLemma}
Let $\psi:\O\ra\RR\cup\{\infty\}$. Then
$\Lphi(\psi)=\psi$ if and only if $\psi=\phi$.
\elem

\bpf
According to Lemma \ref{CalabiLemma}, whenever $(z,w)\in V_\phi$,
$$
2\Re\phiC(z,w)\le \phi(z)+\phi(w),
$$
with equality iff $z=w$. Thus $\Lphi\phi=\phi$.
Conversely, suppose $\Lphi\psi=\psi$. Then, whenever $(z,w)\in V_\phi$,
$$
\psi(z)+\Lphi\psi(w)=
\psi(z)+\psi(w)
\ge 2\Re\phiC(z,w).
$$
Setting $z=w$ gives $\psi\ge \phi$. Since the complex Legendre transform
is order-reversing then also $\psi\le \phi$.
\epf

\bdefin
\label{phiCvxDef}
Say that a function $\psi:\O\ra\RR\cup\{\infty\}$ is $\phi$-convex
if $\psi=\Lphi\eta$
for some usc function $\eta:\O\ra\RR\cup\{\infty\}$.
\edefin

\blem
\label{DoubleDualLemma}
Let $\eta:\O\ra\RR\cup\{\infty\}$ be usc.
Then $\calL^2_\phi\eta\le\eta$ with equality
iff $\eta$ is $\phi$-convex. 

\elem

\bpf
Whenever $(z,w)\in V_\phi$,
$$
\eta(z)+\Lphi\eta(w)\ge 2\Re\phiC(z,w).
$$
Thus,
\beq\label{LtwoetaEq}
\calL^2_\phi\eta(z)
=
\usc\sup_w[2\Re\phiC(w,z)-\Lphi\eta(w)]\le \usc\, \eta(z)=\eta(z).
\eeq
Next, if $\calL^2_\phi\eta=\eta$, then by definition $\eta$
is $\phi$-convex. It remains therefore to show the converse,
and for this
it suffices to show that
$\Lphi^3\nu=\Lphi\nu$ for any usc function $\nu:\O\ra\RR\cup\{\infty\}$.
By  \eqref{LtwoetaEq}, $\Lphi^3\nu\le\Lphi\nu$.
However,
$\calL^2_\phi\nu\le\nu$ by \eqref{LtwoetaEq},
thus $\Lphi^3\nu\ge\Lphi\nu$.

\epf

\section{Legendre duality on compact manifolds}

Remarkably, a variant of these transforms can be defined on any
\K manifold $M$. Let $\omega$ be a closed strictly positive
real-analytic $(1,1)$-form
on $M$.
Locally then $\omega$ equals $\i\ddbar u$ for some strongly psh
real-analytic function $u$, and we define $u_\CC$ and subsequently
$$
D_\omega:=D_u,
$$
locally. To check that these definitions are actually
consistent globally and give rise to a diastasis function on a non-empty neighborhood
 of $\Delta_M$ it suffices to observe
 \cite{Calabi}
that
whenever $h$ is a real-valued
function on a ball in $\CC^n$ that is pluriharmonic, i.e., $\i\ddbar h=0$,
then $h=h_1(p)+\overline{h_1(p)}$ with $h_1$ holomorphic; thus,
$h_{\CCfoot}(p,\b q)=h_1(p)+\overline{h_1(q)}$, so
$D_h\equiv 0$.
Once again, by a variant of Lemma \ref{CalabiLemma} \cite[Proposition 5]{Calabi}
we obtain an open neighborhood $V_\omega$ of the diagonal on which $D_\omega$ is nonnegative and strongly convex in each variable and on which $\omega$
admits local real analytic \K potential $u$ for which $u_\CC$ exists.
This neighborhood contains a $\delta$-tubular neighborhood (with respect to some Riemannian metric) of the diagonal, at least whenever $M$ is compact.

Now, fix a real-analytic \K form $\o$ on $M$.

We can now define a Legendre transform with respect to  $\o$.
\bdefin
\label{cxLegGlobalDef}
The complex Legendre transform $L_\o$ maps a function
$\psi:M\ra\RR\cup\{\infty\}$ to

$$
L_\omega(\psi)(q):=\usc\sup_{p\in M\atop (p,q)\in V_{\ovp}}
[-D_\o(p,q)-\psi(p)].
$$
\edefin

As in the setting of $\O\subset \CC^n$, the transform also depends on
$V_{\ovp}$.

\bdefin
\label{phiCvxDef}
Say that $\psi$ is $\o$-convex
if $\psi=L_\o\eta$ for some usc function $\eta:M\ra\RR\cup\{\infty\}$.
\edefin

The following lemmas follow in the same manner as Lemmas
\ref{SelfDualLemma} and \ref{DoubleDualLemma}.
In fact, intuitively, $L_\o$ is locally given by
$$
L_\o(\psi)(q)=\calL_{u}(u+\psi)-u,
$$
where $u$ is a local \K potential for $\o$.

\blem
\label{SelfDual2Lemm}
Let $\psi:M\ra\RR\cup\{\infty\}$. Then
$L_\o(\psi)=\psi$ if and only if $\psi=0$.
\elem

\blem
\label{DoubleDual2Lemma}
Let $\eta:M\ra\RR\cup\{\infty\}$.
Then $L_\o^2\eta\le\eta$ with equality
iff $\eta$ is $\o$-convex. 

\elem

\section{A generalized gradient map}
\label{GradientSection}

The fact that Calabi's diastasis $D_\o$ is locally uniformly convex
in each variable
on a neighborhood of the diagonal
should, intuitively, ensure
the supremum in the definition of $L_\o$ is attained in a unique point.
In this section we
make this intuition rigorous by giving a condition that ensures the supremum is
attained.
We will discuss only the case of compact manifolds. The case of non compact manifolds, leading up to Theorem 1.1, is proved similarily.

\bthm
\label{CxLegCompactThm}
Let $(M,\o)$ be a compact closed real-analytic \K manifold. There exists $\eps=\eps(\o)>0$ such
that for every function  $\eta$ satisfying $||\eta||_{C^2(M)}<\eps$, the supremum in Definition 4.1 is for any $q$ in $M$ attained at a unique point, $z=G(\eta)(q)$.

\medskip

1. If $\eta$ is of class $C^k$ then $G(\eta)$ is a diffeomorphism of $M$  of class $C^{k-1}$.

\medskip

2. $L_\o(\eta)$ is of class $C^k$ and the map $\eta\to L_\o(\eta)$ is continuous for the $C^k$-topology.

\medskip

3. $L^2_\o(\eta)=\eta$.

\medskip

4. $G(L_\o(\eta))=G(\eta)^{-1}$.

\ethm

\begin{proof}
Because $M $ is compact, the neighborhood $V_\o$ contains a ball of fixed size, call it $\delta>0$ (with respect to the the distance function $d$ of the reference metric $\o$, say), around every point on the diagonal.
Fix $q\in M$.
Let
\begin{equation}
\begin{aligned}
\label{fqEq}
f_q(z)
=
-D_\o(z,q)-\eta(z), \q z\in \pi_1(V_\o\cap (M\times\{q\})).
\end{aligned}
\end{equation}
We claim that $f_q$ attains a unique maximum in $\pi_1(V_\o\cap (M\times\{q\}))$. 
First, Lemma \ref{CalabiLemma} implies that
$$
f_q(z)\le -Cd(z,q)^2-\eta(z), \q z\in \pi_1(V_\o\cap (M\times\{q\})),
$$
and if $||\eta||_{C^2(M)}$ is sufficiently small, $f_q$ is uniformly concave on $\pi_1(V_\o\cap (M\times\{q\}))$.

If $||\eta||_{C^0(M)}<\eps$,
$$
f_q(q)\ge -\eps,
$$
while,
$$
f_q(z)\le -Cd(z,q)^2+\eps, \q z\in \pi_1(V_\vp\cap (M\times\{q\})),
$$
So, if $\eps$ is small enough, 
$$
f_q(z)\le -2\eps, \q z\in
\pi_1(V_\vp\cap (M\times\{q\}))\sm B_{\delta/2}(q),
$$
Thus we see that the maximum of $f_q$ over
$\pi_1(V_\o\cap (M\times\{q\}))$
must be attained at a  point in $B_{\delta/2}(q)$, which moreover is unique by the strict concavity of $f_q$.

This maximum point is the unique solution $z$ of
\beq
\label{Supr2Eq}
F_q(z):=
\nabla_z f_q(z)=
0
\eeq
in $B_{\delta/2}(q)\subset \pi_1(V_\phi\cap (M\times\{q\}))$.
We denote this unique solution
by
$z=G(\eta)(q)$.
Thus, 
\beq 
L_\o(\eta)(q)=f_q(G(\eta)(q))
\label{LphiSupAttainedCptEq}
\eeq
Since $f_q$ is uniformly concave in $B_{\delta/2}(q)$,
the Implicit Function Theorem (IFT) implies that $G(\eta)(q)$
is of class $C^{k-1}$ in $q$ whenever $\eta\in C^k$.
Thus, by \eqref{LphiSupAttainedCptEq} it follows that
$L_\o\psi\in C^{k-1}$.

Next, we claim that $G(\eta)$
is invertible. 
To see that, let
\beq
\label{Supr2Eq}
F_{t,q}(z):=
\nabla f_{t,q}(z),
\q
f_{t,q}(z):=-D_\o(z,q)-t\eta(z).
\eeq
The IFT, applied to $F_{t,q}$, implies that
\begin{equation}
\begin{aligned}
\label{GvptEq}
\nabla G(t\eta)=
-(\nabla_z F_{t,q})^{-1} \nabla_q F_{t,q}
=
(\nabla^2_z D_\o(z,q)-t\nabla^2_z\eta)^{-1} \nabla_q F_{q}
.
\end{aligned}
\end{equation}
When $t=0$, Lemma \ref{SelfDual2Lemm} implies that $G(0)(q)=q$,
so
\begin{equation}
\begin{aligned}
\label{Gvp0Eq}
I
=
\nabla G_\vp(0)
=
-(\nabla^2_z F_{q})^{-1} \nabla_q F_{q}
=
(\nabla^2_z D(z,q))^{-1} \nabla_q F_{q}
.
\end{aligned}
\end{equation}
Combining  \eqref{GvptEq} and \eqref{Gvp0Eq} we see that whenever $||\eta||_{C^2}$
is sufficiently small, the Jacobian of $G(t\eta)$ is positive definite for all $t\in[0,1]$, hence the Jacobian of $G(\eta)$ is invertible. This means that $G(\eta)$ is locally injective, i e that if $q\neq q'$ and $d(q,q')$ is sufficiently small, then $G(\eta)(q)\neq G(\eta)(q')$. Since we moreover know that $G(\eta)$ is uniformly close to the identity, this gives that $G(\eta)$ is globally injective. Since it is also open, it is a diffeomorphism onto its image.

That the supremum in Definition 4.1 is obtained for $z=G(\eta)(q)$ means that
\be
L_\o(\eta)(q)=-D_\o(G(\eta)(q),q)-\eta(G((\eta)(q)).
\ee

Since, for $(z,q)\in V_\o$ we always have
\be
\eta(z)\geq -D_\o(z,q)-L_\o(\eta)(q)
\ee
it follows that
\be
\nabla L_\o(\eta)(q)=-\nabla_q D_\o(z,q).
\ee
when $G(\eta)(q)=z$, so
\be
\nabla L_\o(\eta)(q)=-(\nabla_q D_\o)(G(\eta)(q),q).
\ee
But $D_\o$ is smooth, in fact real analytic, so we get, since $G(\eta)$ is of class $C^{k-1}$  that $\nabla L_\o(\eta)$ is of class $C^{k-1}$ too. In other words $L_\o(\eta)$ is of class $C^k$.

On the other hand, if $\eta$ is close to zero, we know that $G(\eta)$ is close to the identity, which is $G(0)$. Hence, by (5.9), $\nabla L_\o(\eta)$ is close to $\nabla L_\o(0)=0$. Since it follows directly from the definition that the $C^0$-norm of $L_\o(\eta)$ is small if the $C^0$-norm of $\eta$ is small, it follows that $L_\o(\eta)$ is close to zero in the $C^k$-norm if $\eta$ is close to zero in the $C^k$-norm. In particular, with $k=2$, this implies that we can apply the arguments in the beginning of this proof to $L_\o(\eta)$. Then (5.6) implies that $L_\o^2(\eta)=\eta$ and that $G(L_\o(\eta))=G(\eta)^{-1}$.

This completes the proof.
\end{proof}

\section{The inverse gradient map and the complex \MA operator}

As we shall comment later on, our next result can be
viewed as a variant of a result of Lempert.

\bthm
\label{CxGradMapThm}
Fix a real-analytic \K form $\o$.
Then, for each smooth function  $\psi$  such that $\i\ddbar\psi +\o >0$, we let $\omega_\psi:=\o+\i\ddbar\psi$. Then if  $\psi$ satisfies the assumptions of
Theorem \ref{CxLegCompactThm}
$$
G(\psi)^\star \o_\psi = \o_{L_\o\psi}.
$$
Therefore, the complex \MA measure of $\psi$, $\o_\psi^n$, is pulled-back under
$G(\psi)$ to the complex \MA measure of $L_\o\psi$.
\ethm

As pointed out in the introduction and section 2, this should be compared with the  contrasting
fact that the real \MA operator
is pulled-back under the inverse gradient map to the Euclidean measure.

\bpf
Since the statement is local we  look at a neighborhood $W\subset M$ where we
have a real-analytic K\"ahler potential $\phi$ of $\omega$.
By definition, for $(z,w)\in V_\o$
$$
-D_\o(z,w)
\leq \psi(z)+L_\o\psi(w)
$$
with equality precisely when $z=G(\psi)(w)$. Let $\Lambda$ be the set where this holds. Then, when $(z,w)\in\Lambda$, 
$$
-\del_z D_\o(z,w)
=\del_z\psi(z),
$$
or, equivalently, 
$$
-\dbar_w D_\o(z,w)
=
\dbar_wL_\o\psi(w).
$$
In other words (since $D_\o(z,w)=\phi(z)+\phi(w)-\phi_\C(z,w) -\bar\phi_\C(z,w)$),
\beq
\label{delz1Eq}
\del_z\phi_\C (z,w)-\del_z\phi(z)
=\del_z\psi(z),
\eeq
and
\beq
\label{delz2Eq}
\dbar_w\phi_\C(z,w)-\dbar_w \phi(w)
=
\dbar_w L_\o\psi(w).
\eeq

This  means  that the identity holds when both sides are considered
as forms on $\C^{2n}$ and $(z,w)$ lies on $\Lambda$.
Since $\Lambda$ is the graph of $G(\psi)$,  $\Lambda$ is a manifold of  real dimension $2n$.
Let $p_1$ and $p_2$ be the projections of $W\times W$ to the first and second factors, and let $\pi_1$ and $\pi_2$ be their restrictions to $\Lambda$.
By Theorem \ref{CxLegCompactThm},
 $\pi_1$ and $\pi_2$ are invertible maps and
\beq\label{GvpEq}
G(\psi)
=
\pi_1\circ\pi_2^{-1}.
\eeq
By (6.1) and (6.2),
$$
d\phi_\C(z,w)
=
\pi_1^\star\partial(\phi+\psi)(z)
+
\pi_2^\star\dbar(\phi+L_\o\psi)(w), \q \h{when  $(z,w)\in\Lambda$.}
$$
Hence the same identity holds when we restrict both sides to $\Lambda$ as differential forms. 
 Since the left hand side is a closed form, it follows that
$$
\pi_1^*d\partial(\phi+\psi)(z) +\pi_2^*d\dbar(\phi+L_\o\psi)(w)=0, \q \h{on $\Lambda$.}
$$
If we apply $(\pi_2^{-1})^*$ to this equation we get
$$
(\pi_2^{-1})^*\pi_1^*\dbar\partial(\phi+\psi)+\partial\dbar(\phi+L_\o\psi)=0.
$$
By \eqref{GvpEq} it follows that
$$
G(\psi)^\star\omega_\psi=G(\psi)^\star(\i\ddbar(\phi+\psi))=
\i\ddbar(\phi+L_\o\psi)=\omega_{L_\o\psi},
$$
so we are done.
\epf

\section{The Mabuchi metric}

Let $M$ be a closed compact \K manifold.
Recall that if $\omega$ is a K\"ahler form on $M$, the space of $\omega$-plurisubharmonic functions, $\H_\omega$  is the space of smooth functions on $M$ such that $\omega_\psi:=\i\ddbar\psi+\omega>0$. This is an open subset of the space of smooth functions and inherits a structure as a differentiable manifold from the one on $C^\infty(M)$. The tangent space to $\H_\omega$ is the space of smooth functions on $M$ and
one defines a weak Riemannian metric on $\H_\omega$ by
$$
g_M(\nu,\chi)_\psi =\int_M \nu\chi\omega_\psi^n,
$$
for every $\nu,\chi\in T_\psi\H_\o\cong C^\infty(M)$.

\begin{prop}
\label{FirstVarProp}
Let $\o$  be real-analytic.
There exists a neighborhood $\U_\omega$ of 0 in $\calH_\o$ in the $C^2$ topology
such that
$L_\o$ defines a Fr\'echet differentiable map from $\U_\omega$ to $\U_\omega$.
Its differential is
$$
dL_\o(\eta). \chi=-\chi\circ G(\eta), \q \forall \eta\in \U_\omega.
$$
\end{prop}

\bpf
Let $q\in M$.
Define (cf. \eqref{fqEq}),
$$
f_q(z,\eta):=
-D_\o(z,q)-\eta(z), \q z\in \pi_1(V_\o\cap (M\times\{q\})),
$$
and let $F_q(z,\eta):=\nabla_z f_q(z,\eta)$. Then $F_q$ is of class $C^{k-1}$ if $\eta$ is of class $C^k$. By the implicit function theorem
the equation
$$
F_q(z,\eta)=0
$$
defines $z$ as a function of $\eta$, $z=z(\eta)$, and since $z(\eta)$ is the point maximazing $f_q(z,\eta)$ for given $\eta$ we have that $z(\eta)=G(\eta)(q)$ (which we now regard as a function of $\eta$, while $q$ is fixed). Hence we see that $z(\eta)=G(\eta)(q)$ is of class $C^{k-1}$. Moreover
$$
L_\o(\eta)(q)=f_q(z(\eta), \eta).
$$
Hence, by the chain rule
$$
d/dt|_{t=0} L_\o(\eta +t\chi)= d/dt|_{t=0}f_q(z(\eta), \eta+t\chi),
$$
since $\nabla_zf_q(z,\eta)=0$ for $z=z(\eta)$.  Since
$$
d/dt|_{t=0}f_q(z(\eta), \eta+t\chi)=-\chi(z(\eta))=-\chi(G(\eta)(q))
$$
we are done.
\epf

\begin{thm} Let $M$ be a closed compact K\"ahler manifold and let $\omega$ be a real analytic K\"ahler form.
There exists a $C^2$ neighborhood $\U_\o$ of $0$ in $\calH_\o$
such that
$L_\o$ defines a Fr\'echet differentiable map from $\U_\omega$ to
itself with the following properties:

\medskip

\noindent
(i) $L_\o$ is an isometry for the Mabuchi metric on $\U_\o$.

\medskip

\noindent
(ii) $L_\o^2\psi=\psi$ for $\psi\in\U_\o$.

\medskip

\noindent
(iii) $L_\o\psi=\psi$ if and only if $\psi=0$.
\end{thm}

\begin{proof}
Properties (ii) and (iii) are the content of
Lemma \ref{SelfDual2Lemm} and Theorem \ref{CxLegCompactThm}.
We turn to proving (i).
Indeed, by Theorem \ref{CxLegCompactThm} and Proposition \ref{FirstVarProp}
$$
\baeq
g_M(dL_\o\psi(\chi),dL_\o\psi(\nu))|_{L_\o\psi}
&=
\int_M\chi\nu\circ G(\psi)\o^n_{L_\o\psi}
\cr
&=
\int_M\chi\nu\circ G(\psi)G(\psi)^\star(\o^n_{\psi})
\cr
&=
\int_M\chi\nu\o^n_{\psi}
=g_M(\chi,\nu)|_\psi,
\eaeq
$$
proving (i).

Finally, if $\U_\o$ is a neighborhood satisfying properties (i)--(iii),
then replacing $\U_\o$ by
$\U_\o\cap \calL(\U_\o)$, we may assume $\calL$ maps $\U_\o$ to itself.
\end{proof}

This theorem should be seen in the light of the picture of $\H_\omega$ as a symmetric space, put forward by Mabuchi, Semmes and Donaldson,
\cite{M,S,Don}. In these works $\H_\omega$ is first studied as a Riemannian manifold, its curvature tensor is computed and is found to be covariantly constant. In the finite dimensional case, this implies the existence of symmetries around any point in the space. As described in \S\ref{SemmesSubSec}, Semmes has also found
symmetries for the Mabuchi metric,
but to our knowledge the $\omega_\vp$-Legendre transforms are
the first examples of explicit symmetries for $\H_\omega$.
It would be interesting to generalize the theorems of Artstein-Avidan--Milman \cite{AM}
and B\"or\"oczky--Schneider \cite{BS}
to this setting and investigate whether these are {\it all}
the symmetries of $\H_\o$ under some reasonable regularity assumptions.

From Theorem 7.2 it follows in particular that the $\omega$-Legendre transform maps geodesics in $\H_\omega$ to geodesics. By the work of
Semmes \cite{S}, geodesics in $\H_\omega$ are precisely given by solutions of the homogenous complex Monge--Amp\`ere equation, so that a curve
$t\to
\psi_t(z)=\psi(t,z)$
, where $t$ lies in a strip $0<\Re t<1$ and $\psi$ depends only on the real part of $t$, is a geodesic in
$\H_\omega$
if and only if
$$
(\i\ddbar_{t,z}\psi +\omega)^{n+1}=0.
$$
One main motivation  for Lempert's work was to find symmetries of the inhomogenous complex Monge--Amp\`ere equation. Here we find a somewhat different kind of symmetries for the homogenous complex Monge--Amp\`ere equation (HCMA).
The applicability of this may be somewhat limited by the absence of positive existence results for geodesics, but if we change the set up slightly and consider functions $\psi$ defined for $t$ in a disk instead of a strip, there is at least one setting in which our theorem applies. Considering boundary data $s\to\psi_s$ on the unit circle that happen to
extend to a smooth solution of the HCMA,
then the same thing holds for sufficiently small perturbations of the data
\cite{Don2002}, see also \cite{Moriyon}. Taking the given boundary data to be identically equal
to 0, for which trivially an extension exists, we see that any boundary data that are sufficiently small can be extended to a solution of the HCMA, $\psi_t$ with $t$ in the disk $\Delta$. Theorem 7.2 shows that then $L_\o(\psi_t)$ also solves the HCMA.
Indeed, solutions of the HCMA are critical points of the energy functional induced by the Mabuchi metric
$$
E(\psi)=\int_{\Delta\times M} \partial_t\psi\wedge\dbar_t\psi\wedge \omega_{\psi_t}^n,
$$
and as shown above the energy functional is preserved under the $\vp$-Legendre transform.

\section{Relations with Lempert's and Semmes' work}

\subsection{Comparison to Lempert's theorem}

Lempert starts with a complex manifold $M$ and its holomorphic cotangent bundle $T^*(M)$. If $z$ are local coordinates on $M$ it induces local coordinates $(z,\xi)$ on $T^*(M)$, so that a one-form, i e a point in $T^*(M)$ can be written $\sum \xi_j dz_j$. There is a standard holomorphic symplectic form $\Omega$ on $T^*(M)$ that in such coordinates is written $\Omega=\sum d\xi_j\wedge dz_j$. A (local) holomorphic map from $T^*(M)$ to itself, $F$, is symplectic if $F^*(\Omega)=\Omega$, and Lempert's construction depends on the choice of such a symplectic map.
Another ingredient is a
differentiable
real valued function $\psi$ on $M$. From $\psi$ we get a gradient map
\be
z\to(z,\partial\psi)=:\nabla \psi
\ee
which is a section of $T^*(M)$.
Lempert's generalized gradient map is  the map from $M$ to itself
$$
G_\psi= \pi\circ F\circ \nabla\psi,
$$
where $\pi$ is the projection from $T^*(M)$ to $M$. He then defines a generalized Legendre transform by
$$
L_F(\psi)(G_\psi(z))= \psi(z) +2\Re \Sigma(\nabla\psi),
$$
where $\Sigma$ is a {\it generating function} of the symplectic tranformation $F$. This means that $\Sigma$ is holomorphic  on $T^*(M)$ and satisfies
$$
d\Sigma =\xi\cdot dz -F^*(\xi\cdot dz).
$$
Such a generating function exists at least locally since the right hand side is a closed form if $F$ is symplectic.

We indicate briefly how this translates to our set up. First, there is a minor difference that we work with a symplectic form and generating function that is holomorphic in $z$ and antiholomorphic in $\xi$, but the major difference is that we chose a different kind of generating function. The symplectic transformation $F$ gives a map from $T^*(M)$ to $M$ by $w=\pi(F(z))$. For special  symplectic maps (sometimes called {\it free canonical transformations}) one can choose $(z,w)$ as coordinates on $T^*(M)$ and express the generating function in terms of these coordinates instead. Locally, our construction amounts to choosing $\phi_\C(z,w)$ as such a generating function.  If we define a symplectic transformation using $\phi_\C$ as a generating function one can check that our Legendre transform coincides with Lempert's.

\label{SemmesSubSec}
\subsection{Semmes' work}
Another major motivation for our work is Semmes' work \cite{S} and we
now relate the previous theorem to his work.
Semmes starts by endowing the holomorphic cotangent bundle $(T^*)^{1,0}M$
with the complex structure $\hat J$ induced by pulling back the standard complex
structure (induced by the complex strucutre $J$ on $M$) under
the (locally defined) maps \cite[p. 530]{S}
$$
(z,\lambda)\mapsto(z,\lambda+\del_z\phi).
$$
This is well-defined and independent of the choice of local potential $\phi$ for $\o$
since $\del_z\phi-\del_z\phi'$ is holomorphic whenever $\phi'$ is another such choice.
To any smooth \K potential $\psi$, Semmes then associates
the submanifold $\La_\psi$, the graph of $\del\psi$ in $(T^*)^{1,0}M$.
Under the biholomorphism between $((T^*)^{1,0}M,\hat J)$ and $((T^*)^{1,0}M,J)$
 the standard tautological 1-form $\alpha=\sum\la_i dz_i$ and holomorphic symplectic form
$\O=\sum dz_i\w d\la_i$ on the latter are pulled back to forms
that we denote by $\hat \alpha$ and $\hat \O$. Then
$\i\hat\O|_{\La_\psi}=\i\ddbar(\phi+\psi)=\o_\psi$.
Semmes goes on to observe that whenever $\vp$ is real-analytic,
there exists an involutive anti-biholomorphism of a neighborhood of $\La_\vp$
in $((T^*)^{1,0}M,\hat J)$ whose fixed-point set equals $\La_\vp$.
Thus, if $\psi$ is sufficiently close to $\vp$ in $C^2$ then $\La_\psi$ is mapped
to another submanifold that must be of the form $\La_\eta$
for some $\eta$. Theorem \ref{CxGradMapThm} precisely establishes
that this involution is given by our generalized gradient map
$G_\vp(\psi)$, so $G_\vp(\psi)(\La_{L_\o\psi})=\La_\psi$.

\bigskip

\section*{Acknowledgments}
This work is based on the SQuaREs project award
``Interactions between convex geometry and complex geometry"
from the American Institute of Mathematics (AIM).
The authors are grateful to AIM and its staff
for the funding, hospitality, and excellent working conditions
over the years 2011--2013.
YAR is grateful to R.J. Berman and Chalmers Tekniska H\"ogskola
for their hospitality and support in Summer 2014 when an important part of this work was carried out.
Finally, part of this work took place while BB and YAR  visited MSRI (supported by NSF grant DMS-1440140)
during the Spring 2016 semester.
This research was
supported by grants from ANR, BSF (2012236),
ERC, NSF (DMS-0802923,1206284,1515703),
VR, and a Sloan Research Fellowship.

\def\listing#1#2#3{{\sc #1}:\ {\it #2}, \ #3.}

\bigskip

{\sc Chalmers University of Technology and G\"oteborg University}

{\tt bob@chalmers.se}

\bigskip

{\sc Universit\'e Pierre et Marie Curie (Paris 6)}

{\tt cordero@math.jussieu.fr}

\bigskip

{\sc Tel-Aviv University}

{\tt klartagb@post.tau.ac.il}

\bigskip

{\sc University of Maryland}

{\tt yanir@umd.edu}

\end{document}